\documentclass[12pt,a4paper, reqno]{amsart}
\usepackage{amsfonts,amsmath,amsthm,amssymb, amscd}
\usepackage[top=2cm, bottom=5cm, left=3cm, right=1cm]{geometry}
\def\gp#1{\langle#1\rangle}

\usepackage[active]{srcltx}

\def\gp#1{\langle #1 \rangle}
\def\supp#1{\rm supp #1}

\newtheorem{theorem}{Theorem}
\newtheorem{lemma}{Lemma}
\newtheorem{fact}{Fact}
\newtheorem{corollary}{Corollary}


\title{Subgroups   of a finitary linear group}
\author{V.~Bovdi, O.Yu.~Dashkova, M.A.~Salim}

\dedicatory{Dedicated to the 70-th birthday of Professor Yaroslav Sysak}

\thanks{
The research was supported by the UAEU UPAR grant G00002160
}

\address{
V.A.~Bovdi and   M.A.~Salim\\
Department of Math. Sciences\\
UAE University\\
Al-Ain, UAE}
\email{vbovdi@gmail.com, msalim@uaeu.ac.ae}

\address{
O.Yu.~Dashkova\\
Department of Mathematics\\
The Branch of Moscow State University in  Sevastopol\\
Russia}
\email{odashkova@yandex.ua}

\begin{document}
\maketitle

\begin{abstract}
Let $FL_{\nu}(K)$ be the finitary linear group  of degree  $\nu$ over  an  associative   ring $K$   with unity. We prove that  the torsion subgroups of  $FL_{\nu}(K)$      are  locally finite for certain  classes of rings $K$.  A description  of some   f.g. solvable subgroups of   $FL_{\nu}(K)$ are given.
\end{abstract}

\section{Introduction}

Subgroups of the group $GL(F, V)$ of all automorphisms of a vector space $V$ over a field $F$  are called linear groups.  If $V$ has  finite dimension $n$ over $F$ then $GL(F, V)$ is usually denoted  by  $GL_{n}(F)$.
Finite dimensional linear groups  play an important role in various fields of mathematics, physics and natural sciences and have been  studied intensively.  When  $V$ is   an infinite dimensional vector space over $F$, the problem becomes more difficult.  The study of this class of groups  necessitates   certain additional restrictions (see   for example \cite{Dashkova_2, Dashkova_3, Dashkova_1,  Kurdachenko_4,  Dixon_Kurdachenko_Otal, Kurdachenko_5,  Kurdachenko_2, Kurdachenko_3, Kurdachenko_6}).
%It should be noted that in this direction there have been a number of interesting results.
%
%Another important direction in algebra is a study of a group of non-singular  $(n \times n)$-matrices  $GL(F, K)$ %where $K$ is an associative ring.
%%%%%%%%%%%%%%%%%%%%%%%
These types of problems have close relations to the
study of the group of non-singular  $(n \times n)$-matrices  $GL(F, K)$ where $K$ is an associative ring.
%   ????????????????????????
%%%%%%%%%%%%%%%%%%%%%%%%
For an overview  of this topic we  recommend the survey paper
\cite{Vavilov_Stepanov}.

Let $K$ be an  associate  ring with unity and let  $\nu$ be a  linearly ordered set with  order $\leq$. Let $A=(m_{ij}(A))$ be a matrix  of degree $\nu$ over the ring $K$, where   $1 \leq i, j$ and $ i, j \in  \nu$.
Consider all possible subsets $\nu'  \subseteq  \nu$ such that outside $\nu' \times\nu'$  the  matrix  $A$ coincides with the identity  matrix. The intersection of all sets   $ \nu' $  with the given property itself
posesses  this property. Therefore it is the smallest set with such  property. It  is called the support of matrix $A$ and  denoted by $\supp(A)$. Matrices with finite supports are called the {\it  finitary}  matrices. Finitary  matrices are naturally multiplied as $m_{ij}(AB)= \sum_{k}m_{ik}(A)m_{kj}(B)$, where    the sum on the right side contains only  finite numbers of nonzero elements. It is obvious that  $supp(AB)  \subseteq  supp(A)  \cup  supp(B)$. For  all invertible matrixes  $A$   we have $supp(A^{-1}) = supp(A)$.  Hence the set $FL_{\nu}(K)$ of all invertible  finitary matrices of degree  $\nu$ over  $K$ forms  a group under multiplication, and is called   the
 {\it finitary linear group} of degree  $\nu$ over  $K$.

The subgroup $UT_{\nu}(K)$ of $FL_{\nu}(K)$  consisting  all $A\in FL_{\nu}(K)$  with the additional unitriangularity condition        $m_{ij}(A)= \delta_{ij}$ for $i \ge j$ is called the {\it finitary unitriangular  group}.

Finitary  linear groups of degree $\nu$ over a ring  $K$ were introduced  by Yu.I.~Merzlyakov  in \cite{Merzlyakov} and  the same paper  established that $UT_{\nu}(K)$ does not satisfy the normalizer condition for any ring $K$ with  unity and for any infinite linearly ordered set  $\nu$.

The investigation of  $FL_{\nu}(K)$ was started by   \cite{Levchuk} and actively continued in \cite{Kuzucuoglu_Levchuk_2, Kuzucuoglu_Levchuk, Levchuk_Radchenko, Merzlyakov}.

It is known that torsion  subgroups of finite-dimensional linear groups are locally finite (see \cite[Chapter  9]{Wehrfritz}).

The first two results of our paper are  related to  the structure  of  torsion subgroups of $FL_{\nu}(K)$.

\bigskip

\begin{theorem}\label{T:1}
Let   $K$ be   an   integral domain.  Each torsion subgroup $G$  of the finitary linear group  $FL_{\nu}(K)$ is  locally finite.
Moreover, if  $\nu$ is a countable  set,  then $G = \cup_{i \in {\mathbb N}} G_{i}$, where $G_{1} \leq \cdots \leq G_{i} \leq \cdots $  and the following conditions hold:
\begin{itemize}
\item[(i)] each $G_{i}$ contains a normal nilpotent subgroup $N_{i}$ such that  $G_{i}/N_{i}$ is a countable group;

\item[(ii)]  $\prod_{i=1}^{\infty} N_{i}$ is a  subgroup of $G$;

\item[(iii)]  each  group  $N_{1}N_{2} \cdots N_{i}/N_{i}$ is countable.\end{itemize}
\end{theorem}

\bigskip

\begin{theorem}\label{T:2}
Let $K$ be  a   commutative ring. Each torsion subgroup $G$  of the finitary linear group $FL_{\nu}(K)$ is  locally finite.
Moreover, if  $\nu$ is a countable set, then $G = \cup_{i \in {\mathbb N}} G_{i}$, where   $G_{1}\leq \cdots \leq G_{i} \leq \cdots $.

Let    $K$  be a Noetherian commutative ring. Therefore    the following conditions   hold:
\begin{itemize}
\item[(i)] each $G_{i}$ has a series  of normal subgroups $L_{i} \leq M_{i} \leq N_{i} \leq G_{i}$, where    $L_{i}$ is an abelian group, the quotient groups $M_{i}/L_{i}$ and $N_{i} /M_{i}$ are   nilpotent and $G_{i}/N_{i}$ is countable;

\item[(ii)] $\prod_{i=1}^{\infty} N_{i}$ is a  subgroup of $G$;

\item[(iii)]  each quotient   group  $N_{1}N_{2} \cdots N_{i}/N_{i}$ is countable.

\end{itemize}
\end{theorem}

\bigskip

The local structure of  $FL_{\nu}(K)$ is given by the following

\begin{theorem}\label{T:3}
Let   $G$ be  a   f.g.  solvable subgroup of $FL_{\nu}(K)$ over a commutative ring $K$.
\begin{itemize}
\item[(i)]  If $K$ is an  integral domain    then $G$ contains a  normal nilpotent subgroup $N$  such that   $G/N$ is polycyclic.

\item[(ii)]  If $K$  is a  commutative ring  then $G$ contains a series of normal subgroups $L \leq N \leq G$, where    $L$ is  abelian, $N/L$ is  nilpotent-by-nilpotent   and $G/N$ is  polycyclic.

\end{itemize}
\end{theorem}

\bigskip

\begin{corollary}\label{C:1}
Let   $G$ be  a  f.g.  subgroup of $FL_{\nu}(K)$ over a commutative ring $K$.
\begin{itemize}
\item[(i)]  If $K$ is an  integral domain  then  either $G$ contains a  normal nilpotent subgroup $N$ such that $G/N$  is polycyclic-by-finite or  $G$  contains a non-cyclic free subgroup.

\item[(ii)]  If $K$  is a  commutative ring  then  $G$ has a  normal abelian-by-nilpotent subgroups $U$  such  that   either $G/U$ contains a non-cyclic free subgroup or $G$  has a series of normal subgroups $L  \leq U \leq N \leq G$  where $L$ is an  abelian subgroup, $U/L$ and  $N/U$   are nilpotent  and $G/N$ is  polycyclic-by-finite.

\item[(iii)] If $K$ is   a commutative ring and   $G$  is  a   f.g.  subgroup of   $FL_{\nu}(K)$   with  the maximal condition on its subgroups,   then  either  $G$  is a   polycyclic-by-finite  group or $G$   contains a non-cyclic free subgroup.

\end{itemize}
\end{corollary}

\bigskip

\begin{theorem}\label{T:4}
Let $K$ be    a commutative  ring.  Each subgroup of   $FL_{\nu}(K)$  with the minimal condition on its subgroups   is  locally finite.
\end{theorem}

\section{Preliminary results and Proofs}

Let $M$ be  a right $KG$-module, where $KG$ is the group ring of  a group $G$ over a  ring $K$. The {\it centralizer} of $m\in M$ in $G$ is denoteed by $C_G(m)=\{ g\in G\mid mg=m\}$. Let $Aut_{K}(M)$ be the group of $K$-automorphisms  of the module $M$.

Let  $G\leq FL_{\nu}(K)$.  Define
the following right $KG$-module
\[
\mathfrak{A}= \begin{cases}
\bigoplus_{\alpha =1}^{\nu}A_{\alpha } & \text{if $\nu$ is  an   unlimited ordinal  number};\\
  \bigoplus_{\alpha < \nu}A_{\alpha}&\text{  if $\nu$ is a  limited  ordinal number,}
\end{cases}
\]
in which each   $A_{\alpha}$ is isomorphic to the additive group of  $K$ for any ordinal  number  $\alpha$.    In the sequel of our paper  we always assume that
\begin{equation}\label{E:1}
G \not = C_{G}(\mathfrak{A}).
\end{equation}

The properties of the automorphism groups of f.g. modules over commutative rings  play  an important role in the studies of this class of groups (see \cite[Chapter  13]{Wehrfritz}).

\begin{fact}\label{F:1}(\cite[Theorem  13.3]{Wehrfritz})
Let $M$ be a f.g. $K$-module over  a  Noetherian commutative ring $K$. The group   $Aut_{K}(M)$ contains a normal subgroup $U$  stabilizing  a finite series  of submodules of $M$ such that  $Aut_{K}(M)/U$ is  quasi-linear. In particular  $U$ is unipotent and nilpotent (as an abstract group).
\end{fact}

Recall that a group  is called {\it quasi-linear} if it  is isomorphic to a subgroup of a direct product of  the finite number of finite-dimensional linear groups (see \cite[p.\,186]{Wehrfritz}).

We  also use the following  facts.

\begin{fact}\label{F:2}(\cite[Theorem  9.1]{Wehrfritz})
Each  torsion subgroup of $GL_{n}(F)$ is locally finite.
 \end{fact}

\begin{fact}\label{F:3}(\cite[Theorem  9.5]{Wehrfritz})
Each  torsion  linear group is a countable extension of a unipotent (and so nilpotent) group.
 \end{fact}

\bigskip

Our proof starts with the following observation.

\begin{lemma}\label{L:1}
Let $K$ be  a commutative  ring with unity. If  $G$ is  a  f.g. subgroup of $FL_{\nu}(K)$,
then  $G$ contains a normal abelian  subgroup $L$ such that  $G/L$ is isomorphic to a subgroup of the group of automorphisms of a  f.g. module over  a Noetherian commutative  subring of $K$.
\end{lemma}

\begin{proof}
Let $G = \langle g_{1},  \ldots, g_{n} \mid g_{i}\in FL_{\nu}(K)\rangle$.  Set $\nu' = supp( g_{1}) \cup \cdots \cup supp(g_{n})$. Clearly $\nu'$ is a finite set,  $supp (g) \subseteq \nu'$ for each $g \in G$ and $G$ is  isomorphic to a subgroup of $GL_{\nu'}(K)$.
Let $\{a_{1}, a_{2}, \ldots, a_{t}\}$ be the set of all non-zero entries  of the matrices $g_{1},  \ldots, g_{n},  g_{1}^{-1},  \ldots, g_{n}^{-1}\in G$. Let us prove that the commutative ring $K_{1} = \langle a_{1}, a_{2}, \ldots, a_{t} \rangle$ is  Noetherian. Since the subring   $\langle a_{1} \rangle=\mathbb{Z}[a_{1}]\cong \mathbb{Z}[x]$, $\langle a_{1} \rangle$ is   a  commutative Noetherian ring by the Hilbert's theorem.
Using  induction on $k\geq 1$ we have that  $\langle a_{1},\ldots,  a_{k} \rangle=\langle a_{1},\ldots,  a_{k-1} \rangle[a_k]$  is a commutative Noetherian ring by the same argument, so $K_{1}$ is also a commutative Noetherian ring.

Define the right $K_{1}G$-submodule  $B= \bigoplus_{i =1}^{\nu'}A_{i}$ of $\mathfrak{A}$ (see the definition before (\ref{E:1}))  in which  each $A_{i}$ is   isomorphic to the additive group of  $K_{1}$. Since the set $\nu'$ is finite,  $B$ can be considered as a f.g. right $K_{1}G$-module.

Let $C=C_{B}(G)$.  Obviously  $G = C_{G}(C)$ but    $G \not = C_{G}(B)$ by (\ref{E:1}), so
\begin{equation}\label{E:2}
\langle 0 \rangle \leq C \lvertneqq B.
\end{equation}
If $C=\gp{0}$ then  $G/C_{G}(B)$ is a subgroup of the group of automorphisms of the f.g. module $B$ over  $K_{1}$ and  we put  $L=1$. Now let $C\not=\gp{0}$. Set  $L = C_{G}(C) \cap C_{G}(B/C)$. Every  element of  $L$ acts trivially in each  factor of (\ref{E:2}), so $L$ is abelian   (see  \cite[Kaluzhnin's  Theorem,  p.144]{Kargapolov_Merzlyakov}) and  (see \cite[Theorem 4.3.9]{Kargapolov_Merzlyakov})
\[
G/L \leq G/C_{G}(C) \times G/C_{G}(B/C),
\]
so     $G/L$ is isomorphic to a subgroup of $G/C_{G}(B/C)$.

Since  $B$  and $B/C$  are   f.g.  $K_{1}$-modules,   $G/L$ is isomorphic to a  subgroup of the group of automorphisms of the f.g. module $B/C$ over  $K_{1}$. \end{proof}

\begin{proof}[{Proof of Theorem \ref{T:1}}]
(i) Let  $H = \langle g_{1},  \ldots, g_{n}\mid g_i\in G \rangle$ and  $\nu' = supp( g_{1})  \cup \cdots \cup supp(g_{n})$. The group  $H$ is isomorphic to a subgroup of $GL_{\nu'}(K)$.
Since each  integral domain can be embedded in a field,  $H$ is isomorphic to a subgroup of  $GL_{ \nu'} (F)$ for some field $F$.  Torsion  subgroups of $GL_{ \nu'} (F)$   are  locally finite, so   $H$ is finite (see  Fact \ref{F:2}) and  $G$ is a  locally finite group.

Assume   $\nu$ is a countable set and let  $G_{i}$ be   the largest subgroup of $G$, such that  $supp(g) \subseteq \{1, \ldots, i\}$ for any $g \in G_{i}$. Then $G = \cup_{i \in {\mathbb N}} G_{i}$, in which $G_{1} \leq G_{2} \leq \cdots $.
Each $G_{i}$ contains a normal nilpotent subgroup  $N_{i}$ such that  $G_{i}/N_{i}$  is  countable (see  Fact \ref{F:3}).

(ii) Obviously   $N_{2}$ is a normal subgroup of   $G_{2}$, so  $N_{1}N_{2}$ is a  subgroup of  $G_{2}$.  Using induction it is easy to see that   $N_{1}N_{2} \cdots N_{i}\leq G_{i}$   for $2\leq i$. Consequently  we have  an increasing series  of groups $N_{1} \leq N_{1}N_{2} \leq \cdots \leq N_{1}N_{2}\cdots N_{i} \leq \cdots$ , so $\langle N_{1}, N_{2}, \ldots \rangle = \prod_{i=1}^{\infty} N_{i}$.

(iii) Each   $G_{i}/N_{i}$ is countable, so     its subgroup $N_{1}N_{2} \cdots N_{i}/N_{i}$ is countable too.
\end{proof}

\begin{proof}[{Proof of Theorem \ref{T:2}}]
  (i) Let  $H = \langle g_{1},  \ldots, g_{n} \mid g_i\in G \rangle$ and    $\nu' = supp( g_{1}) \cup supp(g_{2}) \cup \cdots \cup supp(g_{n})$.  Moreover  $\nu'$ is a finite set and  $H$ is  isomorphic to a  subgroup of $GL_{\nu'}(K)$. We need to  consider only the case when $H \not = C_{H}(A)$ by  (\ref{E:1}).

The group $H$ contains a normal abelian subgroup $L$ such that  $H/L$ is isomorphic to a subgroup of  a group of automorphisms of a f.g. module over   the Noetherian commutative  ring $K_{1}$  (Lemma \ref{L:1}). The quotient group    $H/L$ is an extension of a nilpotent group by  a quasi-linear group (Fact \ref{F:1}).  As  torsion linear groups are  locally finite,  the  torsion quasi-linear groups are necessarily locally finite (see  Fact \ref{F:2}). Consequently
$H/L$ is finite (see  \cite[Schmidt's  theorem   \S~53]{Kurosh} and    \cite{Schmidt}). Consequently  $H$ is  finite and  $G$ is a  locally finite group.

Let   $\nu$ be  a countable set. As in Theorem \ref{T:1} we prove   that  $G = \cup_{i \in {\mathbb N}} G_{i}$, where $G_{1} \leq G_2  \leq \cdots $. Let $K$  be a Noeterian commutative ring. Then   each   $G_{i}$ has a normal abelian subgroup  $L_{i}$  such that  $ G_{i}/L_{i}$ is an extension of a nilpotent group $M_{i}/L_{i}$ by the  quasi-linear group $(G_{i}/L_{i})/(M_{i}/L_{i})$  (see Lemma \ref{L:1} and   Fact \ref{F:1}).  In view of the isomorphism $ G_{i}/M_{i} \simeq (G_{i}/L_{i})/(M_{i}/L_{i})$ we get that  $G_{i}/M_{i}$ is  quasi-linear. Since a  torsion linear group is a  countable extension of a nilpotent group (see  Fact \ref{F:3}), we obtain  that   a torsion quasi-linear group is a countable extension of a nilpotent group. Therefore, $G_{i}/M_{i}$ contains a normal nilpotent subgroup  $N_{i}/M_{i}$ such that  $G_{i}/N_{i}$ is countable.

(ii)-(iii) The proof is  the same  as of Theorem \ref{T:1}.
\end{proof}

\begin{proof}[{Proof of Theorem \ref{T:3}}]
(i) Since each   integral domain   can be  embedded in a field, the group  $G$ is isomorphic to a f.g. subgroup of $GL_{n}(F)$ for some  field $F$.  Hence   $G$ has a   subgroup of finite index such that its  derived subgroup $N$  is nilpotent   (see \cite[Theorem 3.6]{Wehrfritz}).  Moreover $G/N$  is  polycyclic because   $G$ is a f.g. group.

(ii) If $K$ is a   commutative ring then  $G$ contains a normal abelian  subgroup $L$ such that  $G/L$ is isomorphic to a subgroup of the group of automorphisms of a  f.g. module over  the Noetherian commutative  ring $K_{1}$   (see Lemma \ref{L:1}) and        $G/L$ is an extension of a nilpotent group $S/L$ by the  quasi-linear group $(G/L)/(S/L)$ (see  Fact \ref{F:1}). The isomorphism $(G/L)/(S/L) \simeq G/S$ and the structure of a f.g. solvable quasi-linear group      gives  that $G/S$ has a normal nilpotent subgroup of $N/S$ such that  $(G/S)/(N/S)$ is  polycyclic.
It follows that  $G/N$ is  polycyclic, too.  Since    $N/S$ and $S/L$ are  nilpotent,   $N/L$ is nilpotent-by-nilpotent. Consequently  $G$ contains a series of normal subgroups $L \leq N \leq G$  where $L$ is an  abelian subgroup, $N/L$ is  nilpotent-by-nilpotent and $G/N$ is a  polycyclic group.
\end{proof}

\begin{proof}[{Proof of Corollary \ref{C:1}}]
(i) Each integral domain    can be  embedded in a field. Therefore  $G$ is isomorphic to a  f.g. subgroup of $GL_{n}(F)$ for some field $F$.  Moreover either $G$  has a normal  solvable subgroup $H$ such that $G/H$ is finite  or $G$  contains a non-cyclic free subgroup (see   \cite[Theorem 10.16]{Wehrfritz}).  If $G$  has a normal  solvable subgroup $H$  of finite index then   $H$ has  a  normal nilpotent subgroup $N$  such that   $H/N$ is polycyclic (see  Theorem \ref{T:3}). As $N$ is a characteristic subgroup of $H$ by construction, $N$ is a normal subgroup of $G$. Therefore $G$ has a normal nilpotent  subgroup $N$ such that $G/N$ is polycyclic-by-finite.

(ii) If $K$ is a  commutative ring  then   $G$ has a series of normal subgroups $L \leq U \leq G$ such that $L$ is an abelian subgroup, $U/L$ is nilpotent and $G/U$ is  quasi-linear (see Lemma \ref{L:1} and   Fact \ref{F:1}).  This yields that     either $G/U$  has a normal  solvable subgroup $H/U$  of finite index  or   $G/U$   contains a non-cyclic free subgroup (see the definition of  a quasi-linear group and    \cite[Theorem 10.16]{Wehrfritz}).

Let  $G/U$  has a normal  solvable subgroup $H/U$  of finite index. The group  $H/U$   has a   normal subgroup of finite index such that its  derived subgroup $N/U$  is nilpotent (see \cite[Theorem 3.6]{Wehrfritz}).
Since  $G$ is a f.g. group  and $G/H$ is finite then   $(G/U)/(N/U)$  is  polycyclic-by-finite. It follows that $G$  has a series of normal subgroups
\[
L  \leq U \leq N \leq G
\]
in which  $L$ is  abelian, $U/L$ and  $N/U$   are nilpotent  and $G/N$ is a  polycyclic-by-finite group.

(iii) The group $G$ is either solvable-by-finite or contains a non-cyclic free subgroup (see Corollary \ref{C:1}  (ii)). In the first case the group  $G$ is   polycyclic-by-finite (see
\cite[Theorem 24.1.7] {Kargapolov_Merzlyakov}).
\end{proof}

\begin{proof}[{Proof of Theorem \ref{T:4}}]
If $H$ is  a  f.g. subgroup of $G$, then   either  $H$ is solvable-by-finite group  or $H$ contains a non-cyclic free subgroup (see Corollary \ref{C:1}(i)-(ii)). In the second case $H$ does not  satisfy the minimal condition on subgroups, a  contradiction. Hence $H$ is solvable-by-finite, so   $H$ is a f.g.  Chernikov   group (see \cite[Theorem 24.1.4]{Kargapolov_Merzlyakov}). Consequently  $H$ is  finite.
\end{proof}

The authors are grateful for Professor Denis Osin   for his  valuable remarks.

\end{document}